\documentclass{amsart}
\usepackage{preamble}
\title{Deformation cones of 2-cubical combinatorial zonotopes}

\author{Juliana CURTIS, Grace O'BRIEN, Arnau PADROL, Germain POULLOT}

\makeatletter
\newcommand{\sectionnotoc}[1]{%
  {\let\addcontentsline\@gobblethree
   \section*{#1}}%
}

\AtEndDocument{\bigskip{\footnotesize%
\textsc{University of Washington, United States of America} \\
Juliana Curtis: \texttt{jacca@uw.edu} \hspace{2cm}
Grace O'Brien: \texttt{graceob@uw.edu} \\

\textsc{Universitat de Barcelona \& Centre de Recerca Matem\`atica, Spain} \\
Arnau Padrol: \texttt{arnau.padrol@ub.edu} \\

\textsc{Universit\"at Osnabr\"uck, Germany} \\
Germain Poullot: \texttt{germain.poullot@uni-osnabrueck.de} \\
}}

\begin{document}
\maketitle

\begin{abstract}
We prove that the dimension of the deformation cone of any polytope combinatorially equivalent to a $2$-cubical zonotope is bounded above by the number of generators of the zonotope. This bound is tight, and we show that it cannot be improved to an identity. We further extend our results to any $2$-cubical polytope whose normal fan is combinatorially isomorphic to a pseudohyperplane arrangement, but we show that it cannot be extended to arbitrary $2$-cubical polytopes.
\end{abstract}

\section{Introduction}

\df{Deformations} of a polytope $\pol$ admit many alternative definitions.
They can be equivalently described as the Minkowski summands of dilates of $\pol$, as the polytopes whose normal fans coarsen that of $\pol$, or as those polytopes with the same edge directions as~$\pol$, among others (see \Cref{def:deformation} for details).
The notion of deformation is a fundamental concept in polytope theory that arises naturally in several fields where Minkowski sums play a relevant role (\eg convex geometry, linear optimization, toric varieties, algebraic combinatorics), which explains why all these characterizations have flourished.
As a first introduction to this vast literature, the reader may consult~\cite{Shephard1963,McMullen-typeCone,Meyer,postnikov2008faces,Postnikov2009,DeLoeraRambauSantos,CoxLittleSchenck}. 

The set $\DC(\pol)$ of all deformations of $\pol$ has the structure of a polyhedral cone, called the \df{deformation cone} of $\pol$, which is pointed once translations are modded out. Even though several parameterizations are available for them (see \cite[Appendix 15]{postnikov2008faces}), deformation cones are in general hard to compute. There has been a recent effort by diverse teams to describe deformation cones of families of polytopes of prominent importance in combinatorics~\cite{ACEP-DeformationsCoxeterPermutahedra,CastilloLiu2020, AlbertinPilaudRitter,Castillo_2022,PadrolPaluPilaudPlamondon, PadrolPilaudPoullot-deformedNestohedra,BazierMatteDouvilleMousavandThomasYildirim,PPP25,loho2026rayssubmodularcone}.
Computing the dimension of a deformation cone $\DC(\pol)$ remains challenging.
In fact, deciding whether $\DC(\pol)$ is one-dimensional, which means that $\pol$ is \df{indecomposable}, is a central problem in the area, as indecomposable polytopes are the building blocks for Minkowski decompositions \cite{Gale1954,Shephard1963, kallay1982indecomposable,McMullen1987,PrzeslawskiYost2008,PrzeslawskiYost2016,PinedaVillavicencio,loho2026rayssubmodularcone,padrol2026graphimplicitedgedependencies}.
In general, any polytope normally equivalent to $\pol$ can be decomposed as a Minkowski sum of at most $\dim(\DC(\pol))$ indecomposable summands.
Furthermore, the linear span of $\DC(\pol)$ is also isomorphic to the space of $1$-weights of~$\pol$~\cite{McMullen1996weights} (equivalently, codimension-$1$ Minkowski weights on the normal fan of~$\pol$, see~\cite{FultonSturmfels1997IntersectionTheoryToricVarieties}) and to the first graded piece of the polytope algebra of~$\pol$~\cite{McMullen1989}, and thus their dimensions coincide.

Some of the examples above are \df{zonotopes}, that is, Minkowski sums of line segments. Despite their Minkowski sum presentation, their deformation cones can be very involved. A paradigmatic example is the \df{permutahedron}, which can be described as the zonotope $\pol[\Pi]_n = \sum_{1\leq i < j \leq n} [\b e_i, \b e_j]$ (\eg for $n = 3$, in the plane, $\pol[\Pi]_3$ is a regular hexagon).
Understanding the rays of its deformation cone is a longstanding open problem posed by Edmonds \cite{Edmonds}. 

There is nevertheless a family of zonotopes whose deformation cones are very well behaved: the \df{$2$-cubical zonotopes}, \ie those whose $2$-faces are all parallelograms, or equivalently such that no three of their line segment summands are coplanar.
By~\cite[Example 4.3.6]{padrol2026graphimplicitedgedependencies}, the deformation cone of a $2$-cubical zonotope $\polZ$ is a simplicial cone whose dimension equals the number of segments generating~$\polZ$.

The $d$-cube is a canonical example of a $2$-cubical zonotope, and thus its deformation cone is again simplicial of dimension~$d$. It was proven by Castillo, Doolittle, Goeckner, Ross and Ying in~\cite{Castillo_2022} that the deformation cone of any polytope combinatorially equivalent to the $d$-cube is simplicial of dimension~$d$.
This result is highly surprising because the deformation cone $\DC(\pol)$ heavily depends on the actual geometry of~$\pol$ (parallelisms of edges, angles between the facets, etc.), and not only on its combinatorial type.
Thus, it is natural to wonder whether this also extends to polytopes combinatorially equivalent to other $2$-cubical zonotopes, which we call \df{$2$-cubical combinatorial zonotopes}.
In this direction, we prove in \Cref{sec:result} the following (where the dimension is that of the edge deformation cone, after translations are modded out, see \Cref{rmk:OnDimension}):

\begin{thmUniv}[{\Cref{thm:Aproof}}]\label{thm:combo_zonotope}
Let $\polZ$ be a $2$-cubical zonotope which is the Minkowski sum of $m$ line segments. 
Let~$\pol$ be a polytope combinatorially equivalent to $\polZ$. Then, $\dim(\DC(\pol)) \leq m$.
\end{thmUniv}

Note that this result does not hold for non-$2$-cubical zonotopes. For instance, the aforementioned permutahedron $\pol[\Pi]_n$ is not $2$-cubical (some of its $2$-faces are hexagons), and $\dim(\DC(\pol[\Pi]_n)) = 2^{n} - n - 1$.
Yet $\pol[\Pi]_n$ is the Minkowski sum of $\binom{n}{2} < 2^n - n - 1$ line segments.

For $2$-cubical zonotopes, the bound of \Cref{thm:combo_zonotope} is attained, see \cite[Example~4.2.7, point~4]{padrol2026graphimplicitedgedependencies}. 
Yet, \Cref{thm:combo_zonotope} is in general not an equality. In \Cref{ssec:RhombicDodechedron} we construct a $2$-cubical combinatorial zonotope $\RhombicDodeca$ whose deformation cone $\DC(\RhombicDodeca)$ is not simplicial and has dimension strictly less than its number of generators.
This shows that one cannot extend the results of~\cite{Castillo_2022} to $2$-cubical zonotopes beyond \Cref{thm:combo_zonotope}.

The same example also minimizes $\dim(\DC(\pol))$ among $3$-dimensional polytopes~$\pol$ without triangular faces: indeed, building on Smilansky's \cite[Theorem~4.3]{smilansky1987decomposability}, we prove the following.

\begin{thmUniv}[{\Cref{cor:smil_dc2_two_triangles}}]\label{thmB}
If $\pol$ is a $3$-dimensional polytope with $\dim(\DC(\pol)) = 2$, then $\pol$ must have at least two triangular facets.
In particular, if $\pol$ is a cubical $3$-dimensional polytope, then $\dim(\DC(\pol)) \geq 3$. 
Moreover, if $\pol$ is a $3$-dimensional polytope with no triangular facets and $\dim(\DC(\pol)) = 3$, then~$\pol$ must be a cubical polytope.
\end{thmUniv}

The number of generators of a zonotope coincides with its number of \df{zones}.
In any $2$-cubical polytope (\ie one whose $2$-faces are quadrilaterals), we can define an equivalence relation on the edges by relating all pairs of opposite edges in a quadrilateral $2$-face.
The \df{zones} of a $2$-cubical polytope are the equivalence classes of this relation.
It is natural to ask whether \Cref{thm:combo_zonotope} can be extended to a larger family of $2$-cubical polytopes. 

In \cref{ssec:OrientedMatroidLikePol}, we give a positive answer to this by extending our result from $2$-cubical combinatorial zonotopes (which are dual to hyperplane arrangements) to $2$-cubical pseudozonotopes (which are dual to pseudohyperplane arrangements). Even though this result is strictly stronger than \cref{thm:combo_zonotope}, we kept the proof for zonotopes self-contained so that it is accessible to a broader audience not familiar with oriented matroids.

\begin{thmUniv}[{\Cref{thm:combo_om}}]\label{thmC}
Let $\om$ be a simple $2$-cocubical oriented matroid on a ground set of size $m$, and $\pol$ a polytope combinatorially equivalent to $\om$. Then, $\pol$ has a deformation cone of dimension at most~$m$.
\end{thmUniv}

To conclude, in \Cref{ssec:CubicalNotWork} we show that it is impossible to extend our results to all $2$-cubical polytopes, by giving a $2$-cubical polytope with a single zone but a $3$-dimensional deformation cone.

\sectionnotoc{Acknowledgements} 

This research was partially conducted during the Intensive Research Program on Combinatorial Geometries and Geometric Combinatorics, held at the Centre de Recerca Matemàtica in October-November 2025, and funded by the Severo Ochoa and María de Maeztu Program for Centers and Units of Excellence in R\&D (CEX2020-001084-M), the Institut de Matemàtica de la Universitat de Barcelona, and the Spanish projects PID2022-137283NB-C21 and RED2024-153572-T of MICIU/AEI/10.13039/501100011033.
We want to thank the participants and organizers for many fruitful interactions. In particular, A.P. wants to thank Florian Rieg and Raman Sanyal for many interesting discussions on cubical polytopes.

The research of A.P. is supported by projects PID2022-137283NB-C21 of MICIU/AEI/10.13039/ 501100011033/FEDER UE, the
Spanish-German project COMPOTE (AEI PCI2024-155081-2 \& DFG 541393733), as well as the Mar{\'i}a de Maeztu Program CEX2020-001084-M.

G.O. is supported by the National Science Foundation Graduate
Research Fellowship Program under Grant No. DGE-2140004. Any opinions,
findings, and conclusions or recommendations expressed in this material are those of the
authors and do not necessarily reflect the views of the National Science Foundation.

\vspace{0.25cm}
\sectionnotoc{Use of AI declaration}

The only use we made of AI (namely OpenAI's ChatGPT using GPT-5.5 and Anthropic's Claude using various models) was for English grammar and spelling assistance,
to get ideas on good names for some notions (\eg Claude suggested the terms ``$2$-cubical'' and ``$2$-cocubical''),
to facilitate the encoding of a hand-drawn picture as the TikZ diagram of \cref{fig:colored_zonotope},
and to assist in interpreting and expanding the results presented in \cite{smilansky1987decomposability} leading to \Cref{lem:smil_app}.

Except for this, all statements, proofs, concepts, notations, drawings, references to the literature, etc., were made by us, as humans, using non-AI-powered tools.

\vspace{0.25cm}
\section{Preliminaries}\label{sec:prelim}

We begin by reviewing some basic definitions. 
A \df{polytope} $\pol$ is the convex hull of finitely many points in $\R^n$. 
Letting $\b c \in \R^n$, the \df{face of $\pol$ in the direction of $\b c$} is $\pol^{\b c} \eqdef \{\b x \in \pol : \langle \b x, \b c\rangle = \max_{\b y \in \pol} \langle \b y, \b c \rangle \}$, where $\langle \cdot, \cdot \rangle$ denotes the usual inner product on $\R^n$. 
The \df{face lattice} $\L(\pol)$ is the partially ordered set on the faces of $\pol$, ordered by inclusion. 
Two polytopes are \df{combinatorially equivalent} if their face lattices are isomorphic. 
For a face $\polF$ of $\pol$, its \df{normal cone} is $\Nf_{\pol}(\polF) \eqdef \{\b c \in \R^n : \pol^{\b c} \subseteq \polF \}$. 
The \df{normal fan} of $\pol$, denoted $\Nf(\pol)$, is the collection of the normal cones of its faces. 
In a $d$-dimensional fan, we refer to its $d$-dimensional cones as \df{chambers} and to its $(d-1)$-dimensional cones as \df{walls}.
There is an inclusion-reversing bijection between the faces of $\pol$ and the cones of its normal fan, \eg the edges of a polytope correspond to the walls of its normal fan.

Two polytopes are \df{normally equivalent} (\aka \emph{Minkowski equivalent}) if they have the same normal fan. The \df{Minkowski sum} of two polytopes $\pol$ and $\polQ$ is
$$
    \pol + \polQ = \{\b p + \b q : \b p \in \pol, \b q \in \polQ\}.
$$

We say that $\polQ$ is a \df{weak Minkowski summand} of $\pol$ if there exists a scalar $\lambda > 0$ and a polytope $\polR$ such that $\polQ + \polR = \lambda \pol$. 
Our main motivation is to examine when polytopes can be broken down into nontrivial Minkowski sums. 
    A polytope $\pol$ is \df{indecomposable} if there is no nontrivial Minkowski decomposition of $\pol$, \ie whenever $\pol = \polQ + \polR$ then $\polQ$ and $\polR$ must be (translated) dilates of $\pol$.

Historically, indecomposable polytopes have been studied in various contexts and with various notation, from the original definition of the Minkowski sum by Minkowski in 1911 \cite{Minkowski1911} to the first indecomposability criteria by Gale in 1954 \cite{Gale1954} and the flourishing literature on indecomposability of polytopes \cite{Gale1954,Shephard1963, kallay1982indecomposable,McMullen1987,PrzeslawskiYost2008,PrzeslawskiYost2016,PinedaVillavicencio,loho2026rayssubmodularcone,padrol2026graphimplicitedgedependencies}.
We adopt the presentation and summary of \cite{loho2026rayssubmodularcone}. 
In particular, we use the following definition.
\begin{definition}[{\cite[Definition~2.1]{loho2026rayssubmodularcone}}] \label{def:deformation}
    A polytope $\polQ$ is a \df{deformation} of a polytope $\pol$ %
    if it satisfies any of the following equivalent properties (proved in \cite{postnikov2008faces}).
    \begin{enumerate}[(i)]
        \item \label{item:wms} $\polQ$ is a weak Minkowski summand of $\pol$.
        \item The normal fan of $\polQ$ coarsens the normal fan of $\pol$, \ie every cone $\pol[C] \in \Nf(\polQ)$ is partitioned into cones $\pol[D]_1, \dots, \pol[D]_r\in \Nf(\pol)$ satisfying $\pol[C] = \bigcup_i \pol[D]_i$.
        \item \label{item:edgelengths} $\polQ$ is obtained from $\pol$ by moving the vertices so that the directions of all edges are preserved, \ie there is a (not necessarily injective) map $\b q_i \mapsto \b p_i$ from the vertices of $\polQ$ to the vertices of $\pol$ such that for every edge $\b p_i \b p_j$ of $\pol$, we have $\b q_i - \b q_j = \lambda_{ij} (\b p_i - \b p_j)$ for some $\lambda_{ij} \geq 0$.
    \end{enumerate}
\end{definition}
It follows that a polytope is indecomposable exactly when its only deformations are its (translated) dilates. 

The set of deformations of $\pol$ is the usual and most intuitive definition for what is known as the deformation cone of $\pol$. 
However, there are several linearly isomorphic formulations for the deformation cone, cataloged in \cite[Appendix 15]{postnikov2008faces}. 
Instead of employing the direct definition, we will work with what is sometimes called the \df{edge deformation cone}, defined below, in which translations are already modded out.

\begin{definition}\label{def:DeformationCone}
    Let $E$ be the set of edges of a polytope $\pol$. The \df{deformation cone} of $\pol$, denoted $\DC(\pol)$, is the set of edge-length vectors $\b \lambda \in \R^E$ such that $\lambda_e \geq 0$ for each $e \in E$ and 
    \begin{equation} \label{eq:edge_lengths}
        \lambda_{\b v_1 \b v_2} (\b v_2 - \b v_1) + 
        \lambda_{\b v_2 \b v_3} (\b v_3 - \b v_2) + \dots + 
        \lambda_{\b v_k \b v_1} (\b v_1 - \b v_k) = 0
    \end{equation}
    for every $2$-face of $\pol$ with cyclically ordered vertices $\b v_1, \dots, \b v_k$.
\end{definition}

Thus, the $\b \lambda$ are representations of how each edge of $\pol$ may be scaled while preserving its direction to result in a deformation of $\pol$;
each $\b \lambda$ corresponds to a class of deformations of $\pol$ which are equal up to translation.
Hence, the deformation cone $\DC(\pol)$ is the set of all deformations up to translation.
As the conditions ``$\lambda_e\geq 0$'' and \Cref{eq:edge_lengths} are linear equations and linear inequalities, we see that, as the name indicates, the deformation cone of a polytope is indeed a polyhedral cone.
Consequently, a polytope $\pol$ is indecomposable if and only if its deformation cone is a ray, \ie if and only if $\dim(\DC(\pol)) = 1$.

\begin{remark}\label{rmk:OnDimension}
In the edge deformation cone of $\pol\subseteq \R^d$, the deformations are presented up to translation, hence $\DC(\pol)$ is a pointed cone that does not contain a lineality space.
This differs from the height deformation cone, see \eg \cite{postnikov2008faces,PadrolPaluPilaudPlamondon,PadrolPilaudPoullot-deformedNestohedra,PPP25}, where the same cone is presented with a lineality space of dimension $d$ corresponding to the group of translations.
We insist that the ``$\dim(\DC(\pol))$'' that we discuss throughout this paper is the dimension of the lineality-free cone $\DC(\pol)$ (for instance, it is a lower bound on the number of rays of $\DC(\pol)$).
\end{remark}

Generalizing the search for indecomposability criteria, we are motivated to characterize the dimension of deformation cones.
The above definition leads us to examine the relationships between the edge lengths of a deformation of a polytope $\pol$.
Taking $E$ to be the set of edges of~$\pol$, let $\b \lambda \in \R^{E}_+$ be the edge-length vector corresponding to a deformation of $\pol$.
We say $T \subseteq E$ is \df{dependent on} $S \subseteq E$ if the set $\{ \lambda_{e'} : e' \in T\}$ can be determined from the set $\{ \lambda_{e} : e \in S\}$, meaning that for all $e'\in T$ the exact value of the variable $\lambda_{e'}$ is determined by the system formed by \Cref{eq:edge_lengths} for all $2$-faces of $\pol$ together with the exact values of $\lambda_e$ for $e\in S$.
Further, we say $S \subseteq E$ is a \df{witness for $\pol$} if all of $E$ is dependent on $S$. We say that $S\subseteq E$ is a \df{combinatorial witness} for $\pol$ if for every $\pol'$ combinatorially isomorphic to $\pol$, the set $\phi(S) \subseteq E(\pol')$ is a witness for~$\pol'$, where $\phi$ denotes the combinatorial isomorphism from $\pol$ to $\pol'$.
The names are justified by the following lemma, which puts an upper bound on the dimension of $\DC(\pol)$.

\begin{lemma}\label{lemma:main_idea}
If $S$ is a witness for a polytope $\pol$, then $\dim(\DC(\pol)) \leq |S|$.
Moreover, if it is a combinatorial witness for $\pol$, then $\dim(\DC(\pol')) \leq |S|$ for any $\pol'\cong \pol$.
\end{lemma}
\begin{proof}
Let $\pol$ be a polytope with set of edges $E$, and suppose $S \sbe E$ is a witness for $\pol$.
Let $\b \lambda \in \R^E$ be the edge-length vector corresponding to a deformation of $\pol$.
The $\lambda_e$ for $e \in E \sm S$ are linear functions of the $\lambda_e$ for $e \in S$, and it follows that $\b \lambda$ lies in an $|S|$-dimensional linear subspace of~$\R^E$. 
Thus, $\dim(\DC(\pol)) \leq |S|$.
The extension for combinatorial witnesses follows by definition.
\end{proof}

Our work builds witness sets using the following idea.

\begin{lemma}\label{lem:WitnessForPolygon}
    Let $\pol \sbe \R^2$ be a polygon with set of edges $E$, and let $U$ be a set of two edges of $E$ which are \textbf{not} parallel. Then, the set $E \sm U$ is a witness for $\pol$.
    In particular, if $U$ consists of two consecutive edges, then $E \sm U$ is a combinatorial witness for $\pol$.
\end{lemma}

\begin{proof}
    Suppose $\pol$ has edges $\b e_1, \dots, \b e_k$ cyclically ordered, and suppose that $U = \{ \b e_j, \b e_k\}$ for $j \neq k$.
    Let $\b \lambda = (\lambda_1, \dots, \lambda_k) \in \R^E$ be the edge-length vector corresponding to a deformation $\polQ$ of $\pol$.
    The polygon $\polQ$ has at most $k$ edges and they each correspond to (and are parallel to) an edge of $\pol$:
    The edges of $\polQ$ are $\lambda_{1} \b e_1, \dots, \lambda_{k}\b e_k$ respectively (where $\lambda_{i} = 0$ when $\polQ$ does not have an edge corresponding to $\b e_i$), with $\lambda_i \geq 0$ for all $i$, and by \Cref{eq:edge_lengths},
    $$
        \lambda_{1} \b e_1 + 
        \lambda_{2} \b e_2 + \dots + 
        \lambda_{k} \b e_k = 0 \,.
    $$

    Rearranging the terms, we have that 
    $$
        \lambda_{j} \b e_j + 
        \lambda_{k} \b e_k = 
        - \sum_{i \neq j,k} \lambda_{i} \b e_i \,.
    $$
    By hypothesis, $\b e_j$ and $\b e_k$ are not parallel and so are linearly independent. Thus, given $\bigl\{ \lambda_{i} : i \notin \{j, k\} \bigr\}$, there is a unique choice for $\lambda_{j}$ and $\lambda_{k}$ so that the above equation holds. 
    Therefore, the set $E \sm \{e_j, e_k\}$ is a witness for $\pol$.

    The second claim follows from the fact that consecutive edges of a polygon are never parallel.
\end{proof}

\begin{corollary} \label{cor:infer}
Let $\pol \in \R^2$ be a quadrilateral and $\b e_1, \b e_2$ two adjacent edges of $\pol$, then $\{\b e_1, \b e_2\}$ is a (combinatorial) witness for $\pol$.
\end{corollary}

\begin{proof}
Direct application of \Cref{lem:WitnessForPolygon} to a polygon with four edges.
\end{proof}

In this paper, we will be concerned with the deformation cones of three overlapping categories of polytopes: cubical polytopes, $2$-cubical polytopes, and zonotopes.
A polytope is:
\begin{compactenum}
\item[(a)] \df{cubical} if all of its facets are combinatorially equivalent to (hyper)cubes;

\item[(b)] \df{$k$-cubical} if all its $k$-faces are combinatorially equivalent to (hyper)cubes;

\item[(c)] a \df{zonotope} if it is a Minkowski sum of line segments.
\end{compactenum}

We are interested in \df{$2$-cubical zonotopes} and $2$-cubical polytopes which are combinatorially isomorphic to zonotopes, \ie \df{$2$-cubical combinatorial zonotopes}.
Since $2$-faces of zonotopes are zonotopes, note that if a zonotope is $2$-cubical, then all its $2$-faces are parallelograms: following this line of thought, $2$-cubical zonotopes were named \emph{parallelogrammic zonotopes} in \cite{padrol2026graphimplicitedgedependencies}.
These are the zonotopes such that no three of their defining line segment summands are coplanar.
Note that in three dimensions, the $2$-cubical polytopes are exactly the cubical polytopes, yet this is no longer true in higher dimensions (\eg the prism, in $\R^4$, over the zonotope in \Cref{fig:colored_zonotope} (left) is a non-cubical but $2$-cubical zonotope).
Cubical polytopes will play a role in \Cref{exm:CubicalNotWork} where we show that \Cref{thm:combo_zonotope} does not extend to their setting.

A key observation about deformation cones is that their dimension is not a combinatorial property, \ie two polytopes which are combinatorially equivalent may have deformation cones of different dimensions. One such example can be found in \cite[Example 2.3.7]{padrol2026graphimplicitedgedependencies} (which originally appeared in \cite[Figure 1]{kallay1982indecomposable}), and an even earlier example can be found in \cite[Chapter~3]{Meyer-ThesisMinkowskiAdditionOfConvexSets}.

However, it was shown in \cite{Castillo_2022} that any polytope combinatorially equivalent to a $d$-dimensional cube has a deformation cone of dimension $d$. (The authors moreover prove that the deformation cone of such a combinatorial cube is a simplicial cone.) As both zonotopes and cubical polytopes are natural generalizations of cubes, it is reasonable to explore properties of their deformation cones and to look for an analogous result. 

It was shown in \cite[Example~4.2.7, point~4]{padrol2026graphimplicitedgedependencies} that for a $2$-cubical zonotope $\polZ$ with $m$ (non-parallel) line segment summands, $\dim(\DC(\polZ)) = m$, and that $\DC(\polZ)$ is a simplicial cone.
We wish to understand the dimension of the deformation cones of $2$-cubical combinatorial zonotopes in general.

To work with these combinatorial zonotopes, we require the following additional terminology. 
For a geometric zonotope $\polZ$, its normal fan is equivalent to a \df{hyperplane arrangement}: a finite collection of hyperplanes. 
In this case, it is a collection of $m$ linear hyperplanes whose normal vectors correspond to the line segment summands of $\polZ$, see \cite[Theorem~7.16]{ziegler2012lectures}. 
An example is shown in \Cref{fig:colored_zonotope}. 
It will be useful to define an \df{oriented hyperplane arrangement} as a hyperplane arrangement $\ha$ which has been given an orientation; 
that is, the two open halfspaces of each hyperplane $H \in \ha$ are given a sign and so designated as $H^-$ and $H^+$. 
For an oriented hyperplane arrangement $\ha$, a \df{supertope} of a subset $\c U \sbe \ha$ is the intersection of halfspaces $\c U^+\eqdef \bigcap_{H \in \c U} H^+$ (terminology adopted from \cite{bjorner1999oriented}). If $\c U$ is empty, then $\c U^+$ is the whole ambient space. 
An oriented hyperplane arrangement $\c H$ is \df{acyclic} if there is a base point $\b p$ contained in the positive side of all the halfspaces.
Then, the normal fan of a $2$-cubical zonotope can be expressed as a hyperplane arrangement where no more than two hyperplanes intersect in a $(d-2)$-dimensional subspace, which we will call a \df{$2$-cocubical hyperplane arrangement}.

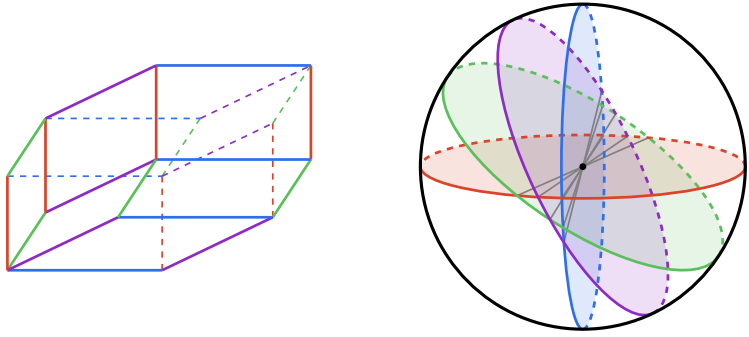
\begin{figure}[h]
    \centering
        \resizebox{0.7\linewidth}{!}{
            \definecolor{zblue}{HTML}{2F6FEB}
\definecolor{zred}{HTML}{D9442B}
\definecolor{zgreen}{HTML}{5BBF5B}
\definecolor{zpurple}{HTML}{8E2BC4}

\begin{tikzpicture}[line width=1.3pt, hidden/.style={gray!45, dashed, line width=0.8pt},
    back/.style={dashed}, disc/.style={fill opacity=0.15},
    meet/.style={gray, line width=0.8pt}]
  \begin{scope}
    \draw[hidden, zred] (2.600,0.000) -- (2.600,1.580); 
    \draw[hidden, zred] (4.460,0.890) -- (4.460,2.470); 
    \draw[hidden, zgreen] (2.600,1.580) -- (3.240,2.550); 
    \draw[hidden, zpurple] (2.600,1.580) -- (4.460,2.470); 
    \draw[hidden, zblue] (0.000,1.580) -- (2.600,1.580); 
    \draw[hidden, zpurple] (3.240,2.550) -- (5.100,3.440); 
    \draw[hidden, zblue] (0.640,2.550) -- (3.240,2.550); 
    \draw[hidden, zgreen] (4.460,2.470) -- (5.100,3.440); 
    \draw[zblue] (0.000,0.000) -- (2.600,0.000); 
    \draw[zgreen] (0.000,0.000) -- (0.640,0.970); 
    \draw[zpurple] (0.000,0.000) -- (1.860,0.890); 
    \draw[zred] (0.000,0.000) -- (0.000,1.580); 
    \draw[zpurple] (2.600,0.000) -- (4.460,0.890); 
    \draw[zgreen] (4.460,0.890) -- (5.100,1.860); 
    \draw[zred] (5.100,1.860) -- (5.100,3.440); 
    \draw[zblue] (2.500,1.860) -- (5.100,1.860); 
    \draw[zblue] (1.860,0.890) -- (4.460,0.890); 
    \draw[zblue] (2.500,3.440) -- (5.100,3.440); 
    \draw[zpurple] (0.640,0.970) -- (2.500,1.860); 
    \draw[zred] (0.640,0.970) -- (0.640,2.550); 
    \draw[zgreen] (1.860,0.890) -- (2.500,1.860); 
    \draw[zred] (2.500,1.860) -- (2.500,3.440); 
    \draw[zgreen] (0.000,1.580) -- (0.640,2.550); 
    \draw[zpurple] (0.640,2.550) -- (2.500,3.440); 
  \end{scope}

  \begin{scope}[shift={(9.67,1.74)}, scale=2.73]
    \fill[zblue, disc, rotate=90.00] (0,0) ellipse (1 and 0.1320);
    \fill[zred, disc, rotate=180.00] (0,0) ellipse (1 and 0.1950);
    \fill[zgreen, disc, rotate=146.58] (0,0) ellipse (1 and 0.3830);
    \fill[zpurple, disc, rotate=115.57] (0,0) ellipse (1 and 0.3300);
    \draw[meet] (0.1295,0.1934) -- (-0.1295,-0.1934); 
    \draw[meet] (0.1223,0.3756) -- (-0.1223,-0.3756); 
    \draw[meet] (0.1168,0.4662) -- (-0.1168,-0.4662); 
    \draw[meet] (-0.4049,-0.1783) -- (0.4049,0.1783); 
    \draw[meet] (-0.2757,-0.1874) -- (0.2757,0.1874); 
    \draw[meet] (-0.2023,-0.3254) -- (0.2023,0.3254); 
    \draw[zblue, back, rotate=90.00] (-1,0) arc[start angle=180, end angle=360, x radius=1, y radius=0.1320];
    \draw[zred, back, rotate=180.00] (-1,0) arc[start angle=180, end angle=360, x radius=1, y radius=0.1950];
    \draw[zgreen, back, rotate=146.58] (-1,0) arc[start angle=180, end angle=360, x radius=1, y radius=0.3830];
    \draw[zpurple, back, rotate=115.57] (-1,0) arc[start angle=180, end angle=360, x radius=1, y radius=0.3300];
    \draw[zblue, rotate=90.00] (1,0) arc[start angle=0, end angle=180, x radius=1, y radius=0.1320];
    \draw[zred, rotate=180.00] (1,0) arc[start angle=0, end angle=180, x radius=1, y radius=0.1950];
    \draw[zgreen, rotate=146.58] (1,0) arc[start angle=0, end angle=180, x radius=1, y radius=0.3830];
    \draw[zpurple, rotate=115.57] (1,0) arc[start angle=0, end angle=180, x radius=1, y radius=0.3300];
    \draw[black, line width=1.5pt] (0,0) circle (1);
    \fill[black] (0,0) circle (0.6pt); 
  \end{scope}
\end{tikzpicture}
        }

        \caption{(Left) Example of a zonotope in $\R^3$ which is the Minkowski sum of four line segments. (Right) Its normal fan, intersected with the sphere $\S^2$ for visualization.}
    \label{fig:colored_zonotope}
\end{figure} 

In \Cref{sec:result,sec:Beyond}, we will find it more convenient to work with the normal fans of polytopes rather than the polytopes themselves, and we will reuse our earlier terminology: 
Letting $\pol$ be a polytope, and
taking $W$ to be the set of walls of $\Nf(\pol)$, we say that $T \subseteq W$ is \df{dependent on} $S \subseteq W$ if the respective corresponding sets of edges $T'$ and $S'$ are such that $T'$ is dependent on~$S'$. 
Further, we say $S \subseteq W$ is a \df{witness for $\pol$} if all of the walls of $\pol$ are dependent on $S$. It is a \df{combinatorial witness} if it is a witness for any polytope in the same isomorphism class.
We will also overload the terminology slightly for a hyperplane arrangement $\c H$ corresponding to the normal fan of a zonotope, where $W$ is the set of all walls contained in the hyperplanes of $\c H$. 
Then, orienting $\ha$, we say a subset $S \subseteq W$ is a \df{witness} for a supertope $R$ of $\c H$ if all of the walls in the interior of $R$ are dependent on the walls in $S$. If this does not depend on the realization, we say that it is a \df{combinatorial witness}.

Then, \cref{cor:infer} can be restated in the form that will be most useful:

\begin{corollary}\label{cor:inferfan}
    For any four walls $\polw_1, \polw_2, \polw_3, \polw_4$ in the normal fan of a polytope, if their common intersection is not contained in any other wall, and if the walls $\polw_1$ and $\polw_2$ are adjacent, then $\{\polw_1, \polw_2\}$ is a combinatorial witness for $\{\polw_1, \polw_2,\polw_3, \polw_4\}$. 
\end{corollary}
\section{Deformation cones of 2-cubical combinatorial zonotopes} \label{sec:result}

In this section we prove our main result, \Cref{thm:combo_zonotope}, which provides an upper bound for the dimension of deformation cones of $2$-cubical combinatorial zonotopes. In \Cref{ssec:RhombicDodechedron}, we show that this is not an equality in general.

\subsection{Proof of our main theorem}
To prove \cref{thm:combo_zonotope}, we wish to examine the dependencies between the walls of the normal fan of a $2$-cubical zonotope, which, as previously stated, can be expressed as a hyperplane arrangement. In particular, we wish to construct a combinatorial witness set of walls using \cref{alg:supertope_witness} below.

We say that a subset $\c V\subseteq \ha$ of an oriented hyperplane arrangement is a \df{huddled set} if every $H\in \c V$ intersects the interior of the supertope $(\ha \setminus \c V)^+$ of its complement. 

\begin{figure}
    \centering
    \includegraphics[width=0.66\linewidth]{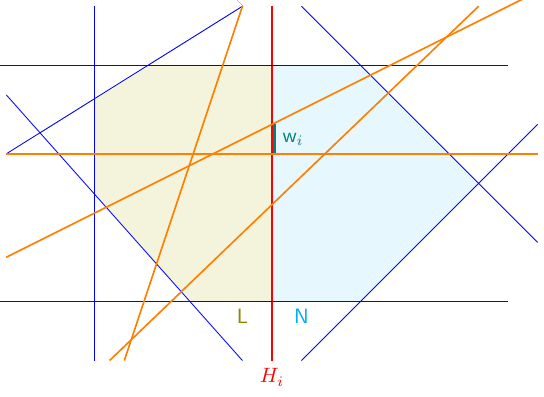}
    \caption{Illustration of the main step of \Cref{alg:supertope_witness} with the notations used in the proof of \Cref{lem:MainAlgoYieldsCombinatorialWitness}.
    In \textcolor{orange}{orange} are the hyperplanes of $\c V_{i-1}$, while in \textcolor{blue}{blue} are the hyperplanes of $\c U_i$.
    The \textcolor{red}{red} hyperplane $H_i$ belongs to $\c V_i$ but not to $\c U_i$.
    We have $\c U_{i-1}^+ = \polL$ and $\c U_i^+ = \polL \cup \polN$.
    }
    \label{fig:Algo}
\end{figure}

\begin{algorithm}[{See \Cref{fig:Algo} for an illustration of the notations}]\label{alg:supertope_witness} \leavevmode

    \textbf{Inputs:}  
    \begin{itemize}
        \item An acyclic oriented hyperplane arrangement $\c H$.
        \item A huddled set $\c V_0 \sbe \c H$ and its complement $\c U_0 \eqdef \c H\sm \c V_0$. 
        \item A set $S_0$ of walls contained in the supertope $\c U_0^+$ of~$\c U_0$. 
    \end{itemize}

    \textbf{Outputs:} 
    \begin{itemize}
        \item Huddled sets $\c V_i\supseteq  \c V_0 $ for $1\leq i\leq |\c U_0|$, with complements $\c U_i \eqdef \c H\sm \c V_i$, together with
         sets $S_i$ of walls of $\c H$ contained in $\c U_i^+$.

    \end{itemize}
    
    \textbf{Steps:} 
    For $1\leq i \leq |\c U_0|$:
    \begin{enumerate}
        \item Choose a wall $\polw_{i}$ which is contained in the boundary of $\c U_{i-1}^+$, and let $H_{i} \in \c U_{i-1}$ be the hyperplane which contains $\polw_{i}$. 
        \item Set $S_{i} \eqdef S_{i-1} \cup \{\polw_{i}\}$, $\c V_{i} \eqdef \c V_{i-1} \cup \{H_{i}\}$ and $\c U_{i} \eqdef \c U_{i-1} \setminus \{H_{i}\}$. 
    \end{enumerate}
\medskip
\end{algorithm}

\begin{lemma}\label{lem:MainAlgoYieldsCombinatorialWitness}
Suppose $\c H$ is an acyclic oriented $2$-cocubical hyperplane arrangement, $\c V_0 \sbe \c H$ is a huddled set with complement $\c U_0 = \c H \sm \c V_0$, and $S_0$ is a combinatorial witness for the walls in the interior of the supertope $\c U_0^+$.  For $1\leq i\leq |\c U_0|$, let $\c V_i$, $\c U_i$ and $S_i$ be the outputs of \cref{alg:supertope_witness}.
Then, $S_i$ is a combinatorial witness set of size~$|S_0|+i$ for the supertope $\c U_i^+$.
\end{lemma}

\begin{proof}  
We proceed by induction on the size $k$ of $\c V_i$. Observe that if $k = 0$ then $i=0$ and the result trivially holds. 
Now, consider $k\geq 1$, and assume as the induction hypothesis that the statement holds for any $S_j'$ and $\c V_j'$ with $|\c V_j'| < k$, when they are obtained with \Cref{alg:supertope_witness} from any valid inputs $\ha'$, $\c V_0'$ and $S_0'$.

If $i=0$ then the statement is trivial. Let $i\geq 1$, and let $\polw_i$ be the wall chosen in Step (1) of the $i^{\text{th}}$ iteration. We will prove that all the walls within $\c U_{i}^+$ are dependent on~$S_i$.
Define $\polN \eqdef H_i^- \cap \c U_{i}^+$ (for `New') and $\polL \eqdef H_i^+ \cap \c U_{i}^+=\c U_{i-1}^+$ (for `oLd'), see \Cref{fig:Algo} for an illustration of the notations. 
Note that the interior walls of $\c U_{i}^+$ are either on $H_i$, or inside $\polL$ or $\polN$.
By hypothesis, $S_{i-1}$ is a combinatorial witness set for 
$\polL$, so every wall in the interior of $\polL$ is dependent on $S_{i-1}\subseteq S_i$.

We now show that every wall of $H_i$ inside $\c U_{i}^+$ is also dependent on~$S_i$, by repeated applications of \Cref{cor:inferfan}. 
Indeed, if $\polw_i$ is the only wall in $H_i \cap \c U_{i}^+$, then we are done because $\polw_i\in S_i=S_{i-1}\cup\{\polw_i\}$. 
If there are other walls in $H_i \cap \c U_{i}^+$, then at least one of them is adjacent to $\polw_i$; call this $\polw'_i$. These two walls are separated by a hyperplane $G$ that must belong to $\c V_i$ since it intersects the interior of $\c U_i^+$ (and is not $H_i$).
The intersection of $\polw_i$ and $\polw_i'$ is thus contained in four walls: $\polw_i$, $\polw_i'$, a wall $\polw''_i\subseteq G$ inside $\polL$, and 
a wall $\polw'''_i\subseteq G$ inside $\polN$.
There are no other walls containing this intersection because the arrangement is $2$-cocubical.
Note that we have already seen that both $\polw_i\in S_i$ and $\polw''_i\subseteq \polL$ are dependent on $S_{i}\supseteq S_{i-1}$. 
Therefore, applying \Cref{cor:inferfan} to $\polw_i$ and $\polw_i''$, we also have that $\polw_i'$ is dependent on $S_i$. 
If there are still remaining walls of $H_i \cap \c U_{i}^+$, at least one of them is adjacent to either $\polw_i$ or $\polw'_i$, so we may repeat this process until we have fixed the lengths of all walls within $H_i \cap \c U_{i}^+$. 
    
In fact, in the same application of \Cref{cor:inferfan}, not only $\polw''_i$ but also $\polw'''_i$ is dependent on~$S_i$.
This holds for every wall adjacent to $H_i\cap\c U_{i}^+$ that is contained in $\polN$, all of which are  dependent on~$S_i$.
We denote by $S'$  this set of walls in $\polN$. We will now prove that $S'$ is a combinatorial witness set for $\polN$.
Let $\c V'$ be the set of hyperplanes containing the elements of $S'$, and let $\c U' = \ha \sm \c V'$. 
Observe that $|S'| = |\c V'| \leq |\c V_i| = k$.
We aim to show that $S', \c V'$ and $\c U'$ are outputs of \cref{alg:supertope_witness}: it remains to choose valid inputs.

Let $\ha'$ be the reorientation of $\ha$ obtained by orienting all hyperplanes towards a base point $\b p'$ which lies in a chamber of $\polN$ that contains a wall of $H_i$. 
Such a point always exists because $H_i$ bounds $\polN$.
Then, we begin by setting $S_0' = \es, \c V_0' = \es$, and $\c U_0 '= \ha'$. We will apply \cref{alg:supertope_witness} restricting our choice of walls $\polw_i$ to those in $S'$.

If $\c V_{j-1}'\neq \c V'$, then there exists a hyperplane $G_j$ on the boundary of $(\c U_{j-1}')^{+}$ that intersects $H_i$. Therefore, there is a wall $\polw_j'$ which is in $S'$ and lies within $G_j$.
We therefore choose $\polw_j'$ in Step~(1), and \cref{alg:supertope_witness} will output $S_{j}' = S_{j-1}' \cup \{\polw_j'\}$, $\c V_{j}' = \c V_{j-1}' \cup \{G_j\}$, and $\c U_{j}' = \c U_{j-1}' \sm \{G_j\}$.
Note that when applying this algorithm iteratively, we have $S_j' \subseteq S'$ for all $j$. So, we will eventually construct $S_m' = S'$ with $\c V_m' = \c V$.
In particular, $S_m$ is the result of $|S'|$ applications of our algorithm. As the input to the $|S'|^{\text{th}}$ application has size $|S'| - 1 < k$, our induction hypothesis implies that $S_m' = S'$ is a combinatorial witness for the supertope $(\c U_m')^+=(\c U')^+$.

In fact, $(\c U')^+=\polN$. Indeed, notice that the hyperplanes of $\c U_i$ do not change orientation when passing from $\ha$ to $\ha'$, because $\b p\in \polL\subseteq\c U_i^+$ and $\b p'\in\polN\subseteq \c U_i ^+$, and all these hyperplanes belong to $\c U'$ by construction.
Also, $H_i$ belongs to $\c U'$, and its orientation in $\ha '$ coincides with the one used to define $\polN$. This shows that $(\c U')^+\subseteq \polN$. This inclusion would be strict if some of the hyperplanes of $\c V_{i-1}$ that cut the interior of $\polN$ belonged to $\c U'$. However, such a hyperplane would also cut $\polL$ (because $\c V_{i-1}$ is huddled by induction), and therefore would also cut $H_i$ and belong to $\c V'$ (not to~$\c U'$) by construction.

Therefore, $S'$, the set of walls in $\polN$ which are adjacent to $H_i$, is indeed a combinatorial witness set for $\polN$.
Hence, since we previously showed that the walls of $S'$ are dependent on $S_i$, we may conclude that all the walls in the interior of $\polN$ are also dependent on~$S_i$.
Putting it all together, we have that the set $S_i$ is a combinatorial witness for all of $\c U_i^+$, as desired. Since all the applications of \cref{cor:inferfan} involve combinatorial witnesses, $S_i$ is indeed a combinatorial witness for $\c U_i^+$.
\end{proof}

We now have the tools we need to prove \Cref{thm:combo_zonotope}.

\begin{theorem}\label{thm:Aproof}
Let $\polZ$ be a $2$-cubical zonotope which is the Minkowski sum of $m$ line segments. 
Let~$\pol$ be a polytope combinatorially equivalent to $\polZ$. Then, $\dim(\DC(\pol)) \leq m$.
\end{theorem}

\begin{proof}
Let $\pol$ be a polytope combinatorially isomorphic to a $2$-cubical zonotope $\polZ$ with $m$ generators, \ie $\Nf(\polZ)$ is induced by a $2$-cocubical arrangement of $m$ hyperplanes.

Choose a base point $\b p$ in a chamber of $\c N (\polZ)$ and let $\ha$ be the oriented $2$-cocubical hyperplane arrangement corresponding to the normal fan, oriented towards $\b p$. Then, set $\c U_0 = \ha$, $\c V_0 = \es$, and $S_0 = \es$. We apply \Cref{alg:supertope_witness} to these inputs. Then, when $\c U_m = \es$ and $\c V_m = \c H$, the output $S_m$ is a combinatorial witness set for (the normal fan of) $\polZ$ of size~$m$ by \Cref{lem:MainAlgoYieldsCombinatorialWitness}.
Since $\pol \cong \polZ$ and $S_m$ is a combinatorial witness, \Cref{lemma:main_idea} implies that $\dim(\DC(\pol)) \leq |S_m| = m$.
\end{proof}

\subsection{A combinatorial rhombic dodecahedron with a 3-dimensional deformation cone with 4 rays}\label{ssec:RhombicDodechedron}
The following example shows that the actual dimension of $\DC(\pol)$ for $\pol$ a $2$-cubical combinatorial zonotope can be strictly smaller than the number of generators of the associated $2$-cubical zonotope.

\begin{example}\label{exm:DCPwithSmallerDim}
A rhombic dodecahedron, see \Cref{fig:colored_zonotope} (left), is a zonotope formed as the Minkowski sum of four line segments in $\R^3$, no three of which are coplanar.
According to \cite[Example~4.2.7, point 4]{padrol2026graphimplicitedgedependencies}, since it is a $2$-cubical zonotope generated by $4$ segments, its deformation cone is a simplicial cone of dimension $4$. 
Additionally, by \Cref{thm:combo_zonotope}, all combinatorial rhombic dodecahedra have a deformation cone of dimension at most $4$.

However, some combinatorial rhombic dodecahedra do not attain this upper bound, as is the case for the polytope $\RhombicDodeca$ in \Cref{fig:combo_rd}: there, a quick computational experiment shows that $\dim(\DC(\RhombicDodeca)) = 3$ and that $\DC(\RhombicDodeca)$ is a cone on $4$ rays.
By exchanging some parallelogrammic faces for skew quadrilaterals, we have removed some symmetries and have decreased the dimension of the associated deformation cone. 

The vertices of $\RhombicDodeca$ are as follows (we obtained them by trial and error, not relying on an AI-powered tool):
$$\begin{array}{ccccccc}\vspace{0.25cm}
(-1, -1, -1)  &  (-1, 1, -1) & (-2, 0, 0) & (0, -2, 0) & (0, 0, -2) & (\tfrac65, -\tfrac45, \tfrac85) & (\tfrac12, \tfrac32, 3) \\ 
(1, 1, -1) & (1, -1, -1) & (2, 0, 0) & (0, 2, 0) & (-1, 1, 2) & (0, 1, 4) & (-\tfrac65, -\tfrac45, \tfrac85)
\end{array}$$
\end{example}

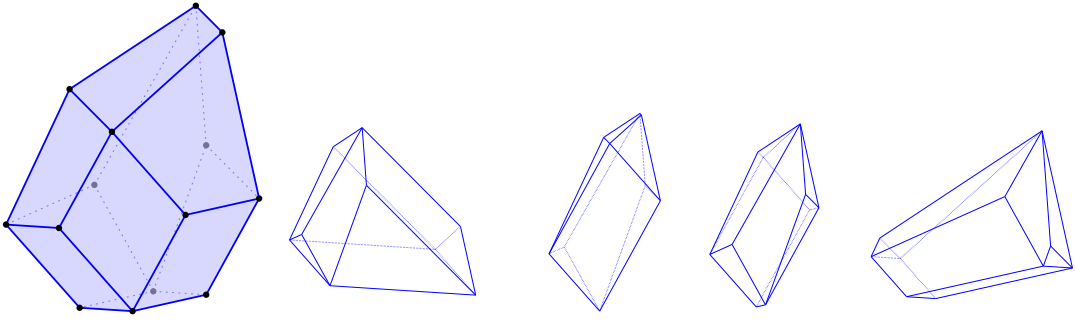
\begin{figure}
    \centering
        \resizebox{0.25\linewidth}{!}{
            \begin{tikzpicture}%
	[x={(-0.161329cm, 0.621664cm)},
	y={(0.986901cm, 0.101660cm)},
	z={(-0.000047cm, 0.776659cm)},
	scale=1.000000,
	back/.style={dotted, thin},
	edge/.style={color=blue!95!black, thick},
	facet/.style={fill=blue!25,fill opacity=0.6},
	vertex/.style={inner sep=1pt,circle,fill=black,thick}]
%
%
\coordinate (-1.20000, -0.80000, 1.60000) at (-1.20000, -0.80000, 1.60000);
\coordinate (-2.00000, 0.00000, 0.00000) at (-2.00000, 0.00000, 0.00000);
\coordinate (-1.00000, -1.00000, -1.00000) at (-1.00000, -1.00000, -1.00000);
\coordinate (-1.00000, 1.00000, -1.00000) at (-1.00000, 1.00000, -1.00000);
\coordinate (-1.00000, 1.00000, 2.00000) at (-1.00000, 1.00000, 2.00000);
\coordinate (0.00000, -2.00000, 0.00000) at (0.00000, -2.00000, 0.00000);
\coordinate (0.00000, 0.00000, -2.00000) at (0.00000, 0.00000, -2.00000);
\coordinate (0.00000, 1.00000, 4.00000) at (0.00000, 1.00000, 4.00000);
\coordinate (0.00000, 2.00000, 0.00000) at (0.00000, 2.00000, 0.00000);
\coordinate (1.00000, -1.00000, -1.00000) at (1.00000, -1.00000, -1.00000);
\coordinate (1.00000, 1.00000, -1.00000) at (1.00000, 1.00000, -1.00000);
\coordinate (0.50000, 1.50000, 3.00000) at (0.50000, 1.50000, 3.00000);
\coordinate (2.00000, 0.00000, 0.00000) at (2.00000, 0.00000, 0.00000);
\coordinate (1.20000, -0.80000, 1.60000) at (1.20000, -0.80000, 1.60000);
\draw[edge,back] (-1.20000, -0.80000, 1.60000) -- (-2.00000, 0.00000, 0.00000);
\draw[edge,back] (-1.20000, -0.80000, 1.60000) -- (0.00000, -2.00000, 0.00000);
\draw[edge,back] (-1.20000, -0.80000, 1.60000) -- (0.00000, 1.00000, 4.00000);
\draw[edge,back] (-2.00000, 0.00000, 0.00000) -- (-1.00000, -1.00000, -1.00000);
\draw[edge,back] (-2.00000, 0.00000, 0.00000) -- (-1.00000, 1.00000, -1.00000);
\draw[edge,back] (-2.00000, 0.00000, 0.00000) -- (-1.00000, 1.00000, 2.00000);
\draw[edge,back] (-1.00000, 1.00000, 2.00000) -- (0.00000, 1.00000, 4.00000);
\draw[edge,back] (-1.00000, 1.00000, 2.00000) -- (0.00000, 2.00000, 0.00000);
\node[vertex] at (-1.20000, -0.80000, 1.60000)     {};
\node[vertex] at (-2.00000, 0.00000, 0.00000)     {};
\node[vertex] at (-1.00000, 1.00000, 2.00000)     {};
\fill[facet] (1.20000, -0.80000, 1.60000) -- (0.00000, 1.00000, 4.00000) -- (0.50000, 1.50000, 3.00000) -- (2.00000, 0.00000, 0.00000) -- cycle {};
\fill[facet] (1.20000, -0.80000, 1.60000) -- (0.00000, -2.00000, 0.00000) -- (1.00000, -1.00000, -1.00000) -- (2.00000, 0.00000, 0.00000) -- cycle {};
\fill[facet] (2.00000, 0.00000, 0.00000) -- (1.00000, 1.00000, -1.00000) -- (0.00000, 2.00000, 0.00000) -- (0.50000, 1.50000, 3.00000) -- cycle {};
\fill[facet] (1.00000, -1.00000, -1.00000) -- (0.00000, -2.00000, 0.00000) -- (-1.00000, -1.00000, -1.00000) -- (0.00000, 0.00000, -2.00000) -- cycle {};
\fill[facet] (1.00000, 1.00000, -1.00000) -- (0.00000, 0.00000, -2.00000) -- (-1.00000, 1.00000, -1.00000) -- (0.00000, 2.00000, 0.00000) -- cycle {};
\fill[facet] (2.00000, 0.00000, 0.00000) -- (1.00000, -1.00000, -1.00000) -- (0.00000, 0.00000, -2.00000) -- (1.00000, 1.00000, -1.00000) -- cycle {};
\draw[edge] (-1.00000, -1.00000, -1.00000) -- (0.00000, -2.00000, 0.00000);
\draw[edge] (-1.00000, -1.00000, -1.00000) -- (0.00000, 0.00000, -2.00000);
\draw[edge] (-1.00000, 1.00000, -1.00000) -- (0.00000, 0.00000, -2.00000);
\draw[edge] (-1.00000, 1.00000, -1.00000) -- (0.00000, 2.00000, 0.00000);
\draw[edge] (0.00000, -2.00000, 0.00000) -- (1.00000, -1.00000, -1.00000);
\draw[edge] (0.00000, -2.00000, 0.00000) -- (1.20000, -0.80000, 1.60000);
\draw[edge] (0.00000, 0.00000, -2.00000) -- (1.00000, -1.00000, -1.00000);
\draw[edge] (0.00000, 0.00000, -2.00000) -- (1.00000, 1.00000, -1.00000);
\draw[edge] (0.00000, 1.00000, 4.00000) -- (0.50000, 1.50000, 3.00000);
\draw[edge] (0.00000, 1.00000, 4.00000) -- (1.20000, -0.80000, 1.60000);
\draw[edge] (0.00000, 2.00000, 0.00000) -- (1.00000, 1.00000, -1.00000);
\draw[edge] (0.00000, 2.00000, 0.00000) -- (0.50000, 1.50000, 3.00000);
\draw[edge] (1.00000, -1.00000, -1.00000) -- (2.00000, 0.00000, 0.00000);
\draw[edge] (1.00000, 1.00000, -1.00000) -- (2.00000, 0.00000, 0.00000);
\draw[edge] (0.50000, 1.50000, 3.00000) -- (2.00000, 0.00000, 0.00000);
\draw[edge] (2.00000, 0.00000, 0.00000) -- (1.20000, -0.80000, 1.60000);
\node[vertex] at (-1.00000, -1.00000, -1.00000)     {};
\node[vertex] at (-1.00000, 1.00000, -1.00000)     {};
\node[vertex] at (0.00000, -2.00000, 0.00000)     {};
\node[vertex] at (0.00000, 0.00000, -2.00000)     {};
\node[vertex] at (0.00000, 1.00000, 4.00000)     {};
\node[vertex] at (0.00000, 2.00000, 0.00000)     {};
\node[vertex] at (1.00000, -1.00000, -1.00000)     {};
\node[vertex] at (1.00000, 1.00000, -1.00000)     {};
\node[vertex] at (0.50000, 1.50000, 3.00000)     {};
\node[vertex] at (2.00000, 0.00000, 0.00000)     {};
\node[vertex] at (1.20000, -0.80000, 1.60000)     {};
\end{tikzpicture}
        }
    \includegraphics[width=0.7\linewidth]{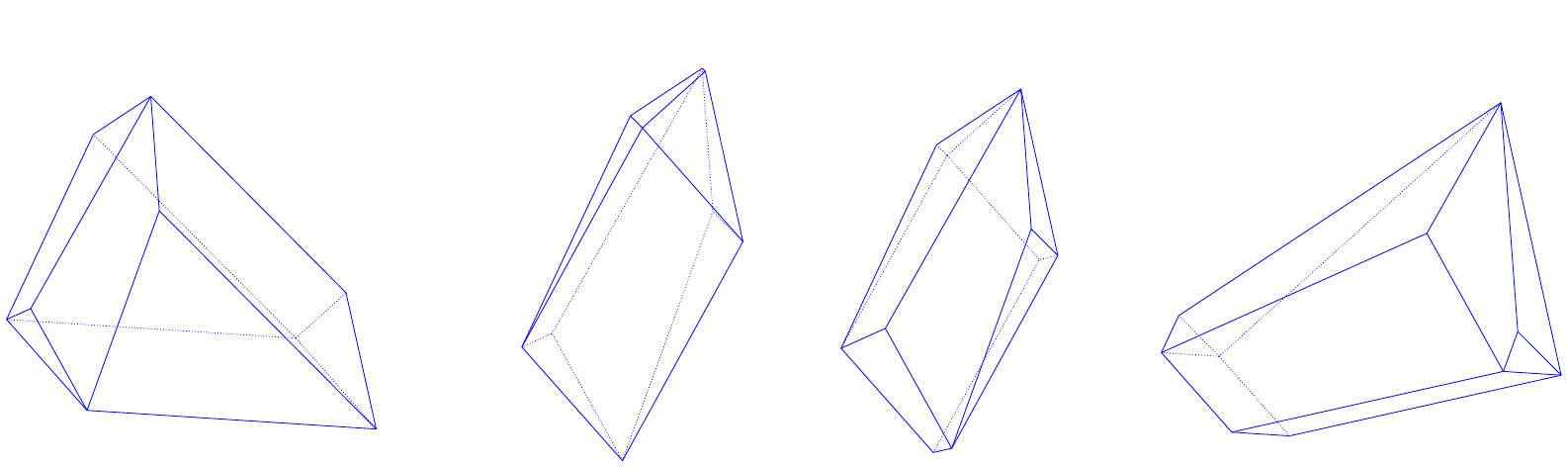}
    \caption{(Left) The combinatorial rhombic dodecahedron $\RhombicDodeca$ whose deformation cone has dimension $3$ and has $4$ rays.
    (Right) The $4$ indecomposable deformations of $\RhombicDodeca$.
    }
    \label{fig:combo_rd}
\end{figure}

The combinatorial rhombic dodecahedron $\RhombicDodeca$ is special for the following three reasons:
\begin{compactenum}
\item $\dim(\DC(\RhombicDodeca)) = 3 < 4$, so it proves that the upper bound given by \Cref{thm:combo_zonotope} is not always attained since the (zonotopal) rhombic dodecahedron has $4$ generators;

\item the number of rays of $\DC(\RhombicDodeca)$ is $4$, so $\DC(\RhombicDodeca)$ is not a simplicial cone, unlike the deformation cones of $2$-cubical zonotopes (see \cite[Example 4.2.7, point 4]{padrol2026graphimplicitedgedependencies});

\item $\RhombicDodeca$ is a cubical $3$-dimensional polytope with a $3$-dimensional deformation cone, which is the lowest possible dimension, as we are going to show in \Cref{cor:smil_dc2_two_triangles}.\label{item:SmilanskyBoundAchieving}
\end{compactenum}

\bigskip

We now discuss the above \eqref{item:SmilanskyBoundAchieving}.
Indeed, via an application of \cite[Theorem~4.3]{smilansky1987decomposability} combined with Euler's formula, we can show that any $3$-polytope without triangular facets cannot have a deformation cone of dimension less than~$3$.
(Note that a polytope without triangular faces can still have 3 vertices pairwise joined by edges.)

\begin{lemma}\label{lem:smil_app}
If $\pol$ is a $3$-dimensional polytope with $t$ triangular facets,
then $\dim(\DC(\pol)) \geq 3 - \frac{t}{2}$.
\end{lemma}

\begin{proof} 
Let $F$, $E$, and $V$ be the sets of facets, edges, and vertices of $\pol$, respectively. We start with a direct application of \cite[Theorem~4.3]{smilansky1987decomposability} to obtain the following.
\begin{align*}
\dim(\DC(\pol))
&\geq |F| - 3 - \sum_{v \in V} \left( \deg(v) - 3 \right)
= |F| - 3 - \left( \sum_{v \in V} \deg(v)\right) + 3 |V| \\
&\geq |F| - 2|E| + 3|V| - 3 
= |V| - |F| + 1,
\end{align*}
where the last line comes from Euler's formula.

Let $s_i$ be the number of edges contained in facet $F_i$.
By double counting: $\sum_i s_i = 2|E|$.
Since $t$ facets have only $3$ edges and the remaining $|F|-t$ facets have at least $4$, we obtain
$
2|E| \geq 3t + 4(|F|-t) = 4|F| - t
$.
Thus, $|E| \geq 2 |F| - \frac{t}{2}$.
Simplifying with Euler's formula: $|V| - |F| + 1 \geq 3 - \frac{t}{2}$.
Substituting into the above expression:
$\dim(\DC(\pol)) \geq 3 - \frac{t}{2}$.
\end{proof}

A well-known result of Smilansky \cite[Corollary~6.8]{smilansky1987decomposability} states that any indecomposable $3$-polytope must have at least $4$ triangular facets, and Smilansky conjectured that there does not exist an indecomposable polytope without triangular $2$-faces in any dimension.
This conjecture is still open, although we do not believe it to be true.
Our lemma above shows that, in fact, there is also a restriction on the number of triangular facets in $3$-dimensional polytopes with deformation cones of higher dimensions.
We now proceed with proving \Cref{thmB}.

\begin{corollary}\label{cor:smil_dc2_two_triangles}
If $\pol$ is a $3$-dimensional polytope with $\dim(\DC(\pol)) = 2$, then $\pol$ must have at least two triangular facets.
In particular, if $\pol$ is a cubical $3$-dimensional polytope, then $\dim(\DC(\pol)) \geq 3$.

Moreover, if $\pol$ is a $3$-dimensional polytope with no triangular facets and $\dim(\DC(\pol)) = 3$, then~$\pol$ must be a cubical polytope.
\end{corollary}

\begin{proof}
The first statement is a direct application of \Cref{lem:smil_app}.

The second statement follows from another application of Euler's formula.
Let $\pol$ be a $3$-dimensional polytope with no triangular facets and with $\dim(\DC(\pol)) = 3$.
For $\pol$, let $F$ be its set of facets, $E$ its set of edges, and $V$  its set of vertices.
As we showed in the proof of \Cref{lem:smil_app}, 
$\dim(\DC(\pol)) \geq |V| - |F| + 1$.
Again, let $s_i$ be the number of edges contained in facet $F_i$. Then, letting $r$ be the number of facets of $\pol$ which contain at least 5 edges, we get
$2|E| = \sum_i s_i \geq 4|F| + r$.
Substituting with Euler's formula and simplifying, we get $|V| - |F| + 1 \geq 3 + \frac{1}{2} r$, so
$\dim(\DC(\pol)) \geq 3 + \frac{1}{2} r$.
Since $\dim(\DC(\pol)) = 3$, it must be that $r = 0$.
Thus, all facets of $\pol$ are quadrilaterals, \ie $\pol$ is cubical.
\end{proof}

\section{Beyond 2-cubical combinatorial zonotopes}\label{sec:Beyond}

Throughout \Cref{sec:result} when discussing hyperplane arrangements, we use only the facts that hyperplanes bisect the ambient space and that they are $2$-cocubical. In particular, we do not rely upon any geometric properties of the hyperplanes. Therefore, our results are not specific to hyperplane arrangements, but rather can be extended to the realm of pseudohyperplane arrangements, \ie oriented matroids.
We will extend our result to any $2$-cubical polytope whose normal fan is combinatorially isomorphic to a pseudohyperplane arrangement. This is a much larger family than combinatorial zonotopes.
Indeed, (graphs of pseudocircle arrangements induced by) simple pseudoline arrangements of rank~$3$ are $2$-cocubical, planar and $3$-connected~\cite{CordovilFukuda1993}, and thus can be realized by a convex $3$-polytope whose dual is $2$-cubical. Most pseudoline arrangements are non-stretchable \cite[Section~8.3]{bjorner1999oriented}, and thus do not come from combinatorial zonotopes.
An explicit example is the arrangement on the front cover of \cite{bjorner1999oriented}, which is a $2$-cocubical oriented matroid that is combinatorially equivalent to a $3$-dimensional polytope yet not to a zonotope.

In \Cref{ssec:OrientedMatroidLikePol}, we introduce the basic notation of oriented matroids and rephrase our results in this language.
In \Cref{ssec:CubicalNotWork}, we show that our main theorem cannot be extended further to all $2$-cubical polytopes with the number of zones as an upper bound, by discussing an example of a cubical $3$-dimensional polytope with a single zone and yet a $3$-dimensional deformation cone.

\subsection{A generalization to oriented-matroid-like polytopes}\label{ssec:OrientedMatroidLikePol}

There are several equivalent formulations for oriented matroids; we will use the covector formulation as presented in \cite{bjorner1999oriented}.

On a ground set $E$, a \df{covector} (\aka \df{sign vector}) is an element of $\{+, -, 0\}^E$.
We denote by $\b 0$ the element $(0, \dots, 0)$.
Also, given a covector $X$, we define the \df{opposite} covector $-X$ such that every coordinate is negated; that is, $(-X)_e = -(X_e)$ for all $e\in E$.
The \df{separation set} of $X$ and $Y$ is $S(X,Y) = \{e \in E : X_e = -Y_e \neq 0 \}$.
The \df{composition} of two covectors $X$ and $Y$ is denoted $X \circ Y$ with, for $e\in E$: $(X \circ Y)_e = X_e$ if $X_e \neq 0$ and $(X \circ Y)_e = Y_e$ otherwise.

\begin{definition} [{\cite[Def~4.1.1]{bjorner1999oriented}}]
    A collection $\om \sbe \{+, -, 0\}^E$ of covectors over a ground set $E$, or the pair $(\om, E)$, is an \df{oriented matroid} if it satisfies the following four covector axioms.
    \begin{enumerate}
        \item [(V0)] $\b 0 \in \om$.
        \item [(V1)] If $X \in \om$ then $-X \in \om$.
        \item [(V2)] If $X,Y \in \om$ then $X \circ Y \in \om$.
        \item [(V3)] If $X,Y \in \om$ and $e \in S(X,Y)$, then there exists $Z \in \om$ such that $Z_e = 0$ and $Z_f = (X \circ Y)_f = (Y \circ X)_f$ for all $f \notin S(X,Y)$.
    \end{enumerate}
\end{definition}

An element $e\in E$ with $X_e=0$ for all $X\in\om$ is called a \df{loop}. Two elements $e,f\in E$ are \df{parallel} if $X_e=X_f$ for all $X\in \om$. An oriented matroid is \df{simple} if it has no loops or parallel elements. From now on, all the oriented matroids that we consider will be simple.%

Every oriented hyperplane arrangement $\ha=\{H_i : i\in E\}$ provides an oriented matroid as follows: for every point $\b p\in \R^n$, construct a covector $\b v\in \{+, -, 0\}^E$ with $v_i=+,-,0$ depending on whether $\b p$ is in the positive side, negative side, or on $H_i$ respectively.
The collection of such covectors is always an oriented matroid, and the oriented matroids obtained this way are called \df{realizable}.
Most oriented matroids are not realizable, see \cite[Section 8.3]{bjorner1999oriented}.

The \df{big face lattice} $\fl(\om)$ of an oriented matroid $\om$ is the partial order on $\om$ in which covectors are compared component-wise using the relations $0 < +$ and $0 < -$ (but $+$ and $-$ incomparable).
If $\om$ is realizable by a hyperplane arrangement $\ha$, then $\fl(\om)$ is isomorphic to the face lattice of the fan induced by $\ha$, \ie to the dual of the face lattice of the associated zonotope.

An oriented matroid is \df{acyclic} if $(+,\dots,+)\in\om$. For $A\subseteq E$ and $X\in \{+,-,0\}^E$, define $_{-A}X$ to be the covector with entries $(_{-A}X)_e=-X_e$ if $e\in A$ and $(_{-A}X)_e=X_e$ otherwise. The oriented matroid $_{-A}\om=\{_{-A}X : X\in \om\}$ is called the \df{reorientation} of $A$ in $\om$. Observe that any loopless oriented matroid has an acyclic reorientation:
It suffices to reorient the set of negative entries of a tope (see below).

We will say that a polytope $\pol$ is \df{combinatorially equivalent} to an oriented matroid $\om$ if the face lattice of its normal fan (\ie the dual of the face lattice of $\pol$) is isomorphic to $\fl(\om)$.
Any zonotope is combinatorially equivalent to the oriented matroid of the associated hyperplane arrangement, but there can also be polytopes combinatorially equivalent to non-realizable oriented matroids (see below).
It turns out that our proof of \Cref{thm:combo_zonotope} extends to any polytope which is combinatorially equivalent to a $2$-cocubical oriented matroid.

It remains to define $2$-cocubical and to translate some of our main concepts to this setting.

The top elements of $\fl(\om)$ are called \df{topes}. They correspond to chambers of a hyperplane arrangement. For every $e\in E$, the set $H(e)\eqdef\{X\in\om\ :\ X_e=0\}$ is called a \df{hyperplane}. The maximal covectors within a hyperplane will be called its \df{walls}.
If $\om$ is loopless, topes correspond to covectors with no zero entries, and walls to those with a single zero entry. Each hyperplane partitions the remaining covectors into two \df{halfspaces} $H^+(e) \eqdef \{X\in\om\ :\ X_e=+\}$ and $H^-(e)\eqdef\{X\in\om\ :\ X_e=-\}$. An intersection of halfspaces is called a \df{supertope}. For $\c U\subseteq E$, we denote by $\c U^+$ the supertope $\bigcap_{e\in \c U}H^+(e)$.
The \df{boundary} of a supertope $\c T$ is the set of walls in~$\c T$ that belong to a single tope in~$\c T$. We say that a hyperplane $H(e)$ intersects the interior of a supertope~$\c T$ if there is a wall in $H(e)$ that is contained in $\c T$ but not in its boundary.
With this language, the notion of \emph{huddled set} extends verbatim to oriented matroids.

We will say that an oriented matroid is \df{$2$-cocubical} if the maximal covectors contained in the intersection of two hyperplanes have exactly two zero entries. With this, we can transfer the proof of \Cref{thm:combo_zonotope} directly to this set-up.
We first translate the algorithm.

\begin{algorithm} \label{alg:supertope_witness_oriented_matroid} \leavevmode

    \textbf{Inputs:}  
    \begin{itemize}
        \item An acyclic oriented matroid $\om$.
        \item A huddled set $\c V_0 \sbe E$ and its complement $\c U_0 \eqdef E\sm \c V_0$. 
        \item A set $S_0$ of walls contained in the supertope $\c U_0^+$ of~$\c U_0$. 
    \end{itemize}

    \textbf{Outputs:} 
    \begin{itemize}
        \item Huddled sets $\c V_i\supseteq  \c V_0 $ for $1\leq i\leq |\c U_0|$, with complements $\c U_i \eqdef E\sm \c V_i$, together with
         sets $S_i$ of walls of $\c H$ contained in $\c U_i^+$.

    \end{itemize}
    
    \textbf{Steps:} 
    For $1\leq i \leq |\c U_0|$:
    \begin{enumerate}
        \item Choose a wall $\polw_{i}$ which is contained in the boundary of $\c U_{i-1}^+$, and let $e_{i} \in \c U_{i-1}$ be such that $\polw_{i} \in H_{e_{i}}$. 
        \item Set $S_{i} \eqdef S_{i-1} \cup \{\polw_{i}\}$, $\c V_{i} \eqdef \c V_{i-1} \cup \{e_{i}\}$ and $\c U_{i} \eqdef \c U_{i-1} \setminus \{e_{i}\}$. 
    \end{enumerate}
\medskip
\end{algorithm}

We now proceed with proving \Cref{thmC}.

\begin{theorem}\label{thm:combo_om}
Let $\om$ be a simple $2$-cocubical oriented matroid on a ground set of size $m$, and $\pol$ a polytope combinatorially equivalent to $\om$. Then, $\dim(\DC(\pol)) \leq m$.
\end{theorem}

\begin{proof}
Up to reorientation, we can suppose that $\om$ is acyclic.
Letting $E$ be the ground set of $\om$, we set $\c U_0 = E$, $\c V_0 = \es$, and $S_0 = \es$ and apply \Cref{alg:supertope_witness_oriented_matroid} to these inputs. We claim that $S_m$ is a combinatorial witness set for the normal fan of $\pol$ of size~$m$. This will conclude the proof by \Cref{lemma:main_idea}.
All the notation below deliberately mimics that of the proof of \Cref{lem:MainAlgoYieldsCombinatorialWitness}, but it applies to an oriented matroid instead.

We follow (almost verbatim) the proof of \Cref{lem:MainAlgoYieldsCombinatorialWitness} and prove by induction on $i$ that after executing \Cref{alg:supertope_witness_oriented_matroid} we obtain for each $0\leq i\leq m$ a combinatorial witness set $S_i$ of size~$i$ for the supertope $\c U_i^+$.
Throughout, we will simplify the notation and identify the covectors of $\om$ with the corresponding cells of $\Nf(\pol)$ via the given combinatorial isomorphism.
In essence, the reader should think of taking a fixed polytope $\pol$ combinatorially equivalent to $\c M$: the combinatorics of the oriented matroid $\c M$ tells us in which order we should perform the applications of \Cref{cor:inferfan}, but we actually perform these applications on the (normal fan of the) polytope $\pol$.

The case $i=0$ is trivial.
If $i\geq 1$, let $\polw_i$ be the wall chosen in the $i^{\text{th}}$ iteration of the algorithm and $e_i$ the element of $E$ such that $\polw_i\in H(e_i)$.
We denote $H_i \eqdef H(e_i)$.
Define $\polN \eqdef H_i^- \cap \c U_{i}^+$ (for `New') and $\polL \eqdef H_i^+ \cap \c U_{i}^+=\c U_{i-1}^+$ (for `oLd').  
Note that the interior walls of $\c U_{i}^+$ are either on $H_i$, or inside $\polL$ or $\polN$.
By the induction hypothesis, $S_{i-1}$ is a combinatorial witness set for $\polL$.

Moreover, every wall of $H_i$ inside $\c U_{i}^+$ is also dependent on~$S_i$. We order the walls of $H_i\cap \c U_{i}^+$ $\polw_i^0, \polw_i^1,\dots, \polw_i^r$ in such a way that every $\polw_i^j$ with $j>0$ is adjacent to at least one of the previous walls (which we can do by the T-convexity of supertopes, see \cite[Section 4.2]{bjorner1999oriented}), and prove it by induction on~$j$.
If there is a single wall $\polw_i$, the claim is straightforward. If $j>0$, then $\polw_i^j$ is adjacent to some $\polw_i^k$ with $k<j$ that is dependent on $S_i$ by induction.
There is a single coordinate $f\neq e_i$ separating $\polw_i^j$ and $\polw_i^k$.
The hyperplane $H(f)$ must belong to $\c V_i$, as it intersects $\c U_i^+$. By the $2$-cocubical property, there are exactly two other walls that contain $\polw_i^j\cap \polw_i^k$.
One of them is in $H^+_i$, hence in $\polL$, and is therefore dependent on $S_i$.
By \Cref{cor:inferfan}, $\polw_i^j$ is also dependent on $S_i$. 

Let $S'$ be the set of walls adjacent to $H_i\cap\c U_{i}^+$ that are contained in $\polN$.
The application of \Cref{cor:inferfan} in the previous step also implies that all walls in $S'$ are dependent on $S_i$. The next step is to prove that $S'$ is a combinatorial witness set for $\polN$.
Let $\c V'$ be the set of hyperplanes containing the elements of $S'$ and $\c U' = E \sm \c V'$.  Observe that $|S'| = |\c V'| \leq |\c V_i| = i$.
   
Let $\om'$ be the reorientation of $\om$ obtained by orienting all hyperplanes towards a tope inside $\polN$ that is incident to $H_i$. 
We set $S_0' = \es, \c V_0' = \es$, and $\c U_0 '= E$ and apply \cref{alg:supertope_witness_oriented_matroid} restricting our choice of walls $\polw_i$ to those in $S'$.

If $\c V_{j-1}'\neq \c V'$, then there exists a hyperplane $G_j$ on the boundary of $(\c U_{j-1}')^{+}$ that intersects $H_i$. Therefore, there is a wall $\polw_j'$ which belongs to $S'$ and lies within $G_j$. We choose this wall $\polw_j'$ in Step~(1) of \cref{alg:supertope_witness_oriented_matroid}.
After $m$ iterations, we have 
$S_m'=S'$. By induction we deduce that $S'$ is a combinatorial witness set for the supertope $(\c U_m')^+=(\c U')^+$.

Now we observe that $(\c U')^+=\polN$. If this were not the case, $\c U'$ would contain some element $f\in\c V_{i-1}$ whose associated hyperplane $H(f)$ has a wall in $\polN$.
However, such a hyperplane $H(f)$ would also have a wall in $\polL$ (because $\c V_{i-1}$ is huddled by induction), and therefore would also have a covector in $H_i$, consequently $f\in \c V'$ (so $f\notin \c U'$) by construction.

We conclude that $S'$ is a combinatorial witness for $\polN$.
Since the walls in $S'$ are dependent on $S_i$, we get that all the walls in the interior of $\polN$ are also dependent on $S_i$.
This shows that $S_i$ is a combinatorial witness for all of $\c U_i^+$ and concludes the proof.
\end{proof}

\subsection{A cubical 3-polytope with a 3-dimensional deformation cone but only 1 edge class}\label{ssec:CubicalNotWork}

After working with $2$-cubical combinatorial zonotopes, and extending to $2$-cubical pseudozonotopes, it is natural to wonder if our results extend to cubical polytopes.
Sadly, no.

Let $\pol$ be a $2$-cubical polytope (\eg a cubical polytope).
Two edges of $\pol$ are \df{opposite} if they belong to a common $2$-face and are not adjacent.
The \df{edge classes} of $\pol$ are the equivalence classes of (the transitive closure of) the relation ``being opposite in a quadrilateral face''.
Note that, in a zonotope (resp.\ in a $2$-cocubical oriented matroid), the edge classes correspond to classes of parallel edges (resp.\ to walls on the same hyperplane), and the number of edge classes is precisely the number of generators of the zonotope (resp.\ the number of elements in the ground set).
Hence, one may hope that if $\pol$ is a cubical polytope or even a $2$-cubical polytope, then $\dim(\DC(\pol))$ is bounded above by the number of edge classes of $\pol$.
This is false, as shown by the following example.

\begin{example}\label{exm:CubicalNotWork}
Consider the tetragonal trapezohedron $\pol[Trap]$ drawn in \Cref{fig:tet_trap}.
It is a $2$-cubical $3$-dimensional polytope.
It has only $1$ edge class, yet it is not indecomposable.
By \Cref{lem:smil_app}, since $\pol[Trap]$ has no triangle, $\dim(\DC(\pol[Trap])) \geq 3$.
A quick computation shows that $\dim(\DC(\pol[Trap])) = 3$.
\end{example}

\begin{figure}[H]
    \centering
    \includegraphics[width=0.2\linewidth]{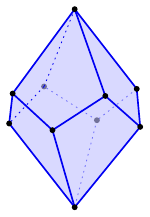}\hspace{1.5cm}%
    \includegraphics[width=0.35\linewidth]{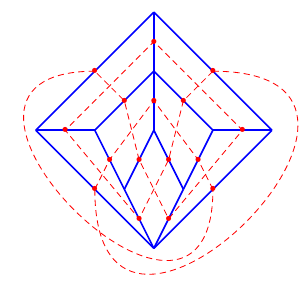}
    \caption{(Left) The tetragonal trapezohedron $\pol[Trap]$ whose deformation cone is $3$-dimensional.
    (Right) A planar representation of (the graph of) $\pol[Trap]$ where edges opposite in a quadrilateral are linked by a \textcolor{red}{red} arc: there is a single class.
    }
    \label{fig:tet_trap}
\end{figure}

\printbibliography

\end{document}